\documentclass[11pt]{amsart}
\usepackage{bbm,amsfonts,latexsym,amsmath,amscd,amssymb,fancyhdr}
\usepackage[all]{xy}

\theoremstyle{plain}
\newtheorem{theorem}{Theorem}
\numberwithin{theorem}{section}

\numberwithin{corollary}{section}

\numberwithin{definition}{section}

\newtheorem{lemma}{Lemma}
\numberwithin{lemma}{section}

\numberwithin{proposition}{section}

\numberwithin{remark}{section}

\numberwithin{example}{section}

\numberwithin{equation}{section}

\newcommand {\be}{\begin{equation}}
\newcommand {\ee}{\end{equation}}

\newcommand{\h}{\begin{eqnarray*}}
\newcommand{\e}{\end{eqnarray*}}

\newcommand{\CC}{\mathbf{C}}
\newcommand{\RR}{\mathbf{R}}
\newcommand{\ZZ}{\mathbf{Z}}
\newcommand{\HH}{\mathbf{H}}

\newcommand{\ii}{\sqrt{-1}}

\begin{document}

\title[$E_8$ Bundles and Rigidity]{$E_8$ Bundles and Rigidity}
\author{Fei Han}
\address{Fei Han, Department of Mathematics, National University of Singapore,
 Block S17, 10 Lower Kent Ridge Road,
Singapore 119076 (mathanf@nus.edu.sg)}
\author{Kefeng Liu}
\address{Kefeng Liu, Department of Mathematics, University of California at Los Angeles,
Los Angeles, CA 90095, USA (liu@math.ucla.edu) and  Center of Mathematical Sciences, Zhejiang University, 310027, P.R. China}
\author{Weiping Zhang}
\address{Weiping Zhang, Chern Institute of Mathematics \& LPMC, Nankai
University, Tianjin 300071, P.R. China. (weiping@nankai.edu.cn)}
\maketitle

\begin{abstract} In this paper, we establish rigidity and vanishing theorems for Dirac operators twisted by $E_8$ bundles.

\end{abstract}

\section*{Introduction}
Let $X$ be a closed smooth connected manifold which admits a nontrivial $S^1$ action. Let $P$ be an elliptic differential operator on $X$ commuting with the $S^1$ action. Then the kernel and cokernel of $P$ are finite dimensional representation of $S^1$. The equivariant index of $P$ is the virtual character of $S^1$ defined by
\be \mathrm{Ind}(g, P)=\mathrm{tr}|_g\mathrm{ker}P-\mathrm{tr}|_g\mathrm{coker}P,\ee
for $g\in S^1$. We call that $P$ is $rigid$ with respect to this circle action if $\mathrm{Ind}(g, P)$ is independent of $g$.

It is well known that classical operators: the signature operator for oriented manifolds, the Dolbeault operator for almost complex manifolds and the Dirac operator for spin manifolds are rigid \cite{AH1}. In \cite{W}, Witten considered the indices of Dirac-like operators on the free loop space $LX$. The Landweber-Stong-Ochanine elliptic genus (\cite{LS}, \cite{Och1})  is just the index of one of these operators. Witten conjectured that these elliptic operators should be rigid. See \cite{La} for a brief early history of the subject.  Witten's conjecture were first proved by Taubes \cite{T} and Bott-Taubes \cite{BT}. Hirzebruch \cite{Hir1} and Krichever \cite{Kr} proved Witten's conjecture for almost complex manifold case. Various aspects of mathematics are involved in these proofs. Taubes used analysis of Fredholm operators, Krichever used cobordism, Bott-Taubes and Hirzebruch used Lefschetz fixed point formula. In \cite{Liu1, Liu2}, using modularity, Liu gives simple and unified proof as well as various generalizations of the Witten conjecture. Several new vanishing theorems are also found in \cite{Liu1, Liu2}. Liu-Ma \cite{LM1, LM2} and Liu-Ma-Zhang \cite{LMZ1, LMZ2} established family versions of rigidity and vanishing theorems. 

In this paper, we study rigidity and vanishing properties for Dirac operators twisted by $E_8$ bundles. Let $X$ be an even dimensional closed spin manifold and $D$ the Dirac operator on $X$. Let $P$ be an (compact-)$E_8$ principal bundle over $X$. Let $W$ be the vector bundle over $X$ associated to the complex adjoint representation $\rho$ of $E_8$. The twisted Dirac operator $D^W$ plays a prominent role in string theory and $M$ theory. In \cite{W1}, the index of such twisted operator is discovered as part of the phase of the $M$-theory action. In \cite{DMW}, the partition function in $M$-theory, involving the index theory of an $E_8$ bundle, is compared with the partition function in type IIA string theory described by $K$-theory to test $M$-theory/Type IIA duality. In this paper, we are interested in the equivariant index of the operator $D^W$ and establish rigidity and vanishing theorems for this operator.

More precisely, let $X$ be a $2k$ dimensional closed spin manifold, which admits a nontrivial $S^1$ action. Let $P$ be an (compact-)$E_8$ principal bundle over $X$ such that the $S^1$ action on $X$ can be lifted to $P$ as a left action which commutes with the free action of $E_8$ on $P$. Let $W$ be the complex vector bundle associated to the complex adjoint representation of $E_8$ mentioned above. Then  the $S^1$ action on $P$ naturally induces an action on $W$ by $g\cdot[s, v]=[g\cdot s, v]$, where $[s, v]$ with $s\in P, v\in \CC^{248}$, is the equivalent classes defining the elements in $W$ by the equivalent relations $(s, v)\sim(s\cdot h, \rho(h^{-1})\cdot v)$ for $h\in E_8$. Let $X^{S^1}$ be the fixed point manifold and $\pi$ be the projection from $X^{S^1}$ to a point $pt$. Let $u$ be a fixed generator of $H^2(BS^1, \ZZ)$. We have the following theorem:

\begin{theorem} Assume the action only has isolated fixed points and the restriction of the equivariant characteristic class $\frac{1}{30}c_2(W)_{S^1}-p_1(TX)_{S^1}$ to $X^{S^1}$ is equal to $n\cdot \pi^*u^2$ for some integer $n$. \newline
(i) If $n<0$, then $\mathrm{Ind}(g, D^W)$ is independent of $g$ and equal to $-\mathrm{Ind}(D^{T_\CC X})$, minus the index of the Rarita-Schwinger operator. In particular, one has $\mathrm{Ind}D^W$=$-\mathrm{Ind}D^{T_\CC X}$ and when $k$ is odd, i.e. dim $X\equiv 2 \,(mod \, 4)$, one has $\mathrm{Ind} (g,D^W)\equiv 0.$\newline
(ii) If $n=0$, then $\mathrm{Ind}(g, D^W)$ is independent of $g$. Moreover, when $k$ is odd, one has $\mathrm{Ind}(g, D^W)\equiv 0$.\newline
(iii) If $n=2$ and  $k$ is odd, then $\mathrm{Ind}(g, D^W)\equiv 0$.

\end{theorem}

Actually we have established rigidity and vanishing results in more general settings concerning the twisted spin$^c$ Dirac operators.  See  Theorem 2.1 and Theorem 2.2 for details. The above theorem is a corollary of Theorem 2.1. We prove our theorems by studying the modularity of Lefschetz numbers of certain elliptic operators involving the basic representation of the affine Kac-Moody algebra of $E_8$. In the rest of the paper, we will first briefly review the Jacobi theta functions and the basic representation for the affine $E_8$ by following \cite{K1} (see also \cite{K2}) as the preliminary knowledge in Section 1 and then state our theorems as well as give their proofs in Section 2.

\section{Preliminaries}
\subsection{Jacobi theta functions}
The four Jacobi theta-functions are defined as follows (cf. \cite{C}),

\begin{equation}
\theta (z,\tau )=2q^{1/8}\sin (\pi z)\prod_{j=1}^{\infty
}[(1-q^{j})(1-e^{2\pi \sqrt{-1}z}q^{j})(1-e^{-2\pi \sqrt{-1}z}q^{j})],
\end{equation}%
\begin{equation}
\theta _{1}(z,\tau )=2q^{1/8}\cos (\pi z)\prod_{j=1}^{\infty
}[(1-q^{j})(1+e^{2\pi \sqrt{-1}z}q^{j})(1+e^{-2\pi \sqrt{-1}z}q^{j})],
\end{equation}%
\begin{equation}
\theta _{2}(z,\tau )=\prod_{j=1}^{\infty }[(1-q^{j})(1-e^{2\pi \sqrt{-1}%
z}q^{j-1/2})(1-e^{-2\pi \sqrt{-1}z}q^{j-1/2})],
\end{equation}

\begin{equation}
\theta _{3}(z,\tau )=\prod_{j=1}^{\infty }\left[ (1-q^{j})(1+e^{2\pi \sqrt{%
-1}z}q^{j-1/2})(1+e^{-2\pi \sqrt{-1}z}q^{j-1/2})\right] ,
\end{equation}%
where $q=e^{2\pi \sqrt{-1}\tau },\ \tau \in \mathbf{H}$, the upper half
plane.

They are all holomorphic functions for $(z,\tau )\in \mathbf{C\times H}$,
where $\mathbf{C}$ is the complex plane.

Let $\theta ^{\prime }(0,\tau )=\frac{\partial }{\partial z}\theta (z,\tau
)|_{z=0}$. One has the following Jacobi identity (c.f. \cite{C}),
\begin{equation}
\theta ^{\prime }(0,\tau )=\pi \theta _{1}(0,\tau )\theta _{2}(0,\tau
)\theta _{3}(0,\tau ).
\end{equation}

Let
\begin{equation*}
SL(2,\mathbf{Z}):=\left\{ \left. \left(
\begin{array}{cc}
a_{1} & a_{2} \\
a_{3} & a_{4}%
\end{array}%
\right) \right\vert a_{1},a_{2},a_{3},a_{4}\in \mathbf{Z},\
a_{1}a_{4}-a_{2}a_{3}=1\right\}
\end{equation*}
be the modular group. Let $S=\left(
\begin{array}{cc}
0 & -1 \\
1 & 0%
\end{array}%
\right) ,\ T=\left(
\begin{array}{cc}
1 & 1 \\
0 & 1%
\end{array}%
\right) $ be the two generators of $SL(2,\mathbf{Z})$. Their actions on $%
\mathbf{H}$ are given by
\begin{equation*}
S:\tau \mapsto -\frac{1}{\tau },\ \ \ T:\tau \mapsto\tau +1.
\end{equation*}

The actions on theta-functions by $S$ and $T$ are given by the following transformation
formulas (cf. \cite{C}),
\begin{equation}
\theta (z,\tau +1)=e^{\frac{\pi \sqrt{-1}}{4}}\theta (z,\tau ),\ \ \ \theta
\left( z,-{1}/{\tau }\right) ={\frac{1}{\sqrt{-1}}}\left( {\frac{\tau }{%
\sqrt{-1}}}\right) ^{1/2}e^{\pi \sqrt{-1}\tau z^{2}}\theta \left( \tau
z,\tau \right) \ ;  \label{Eqn: theta transformation}
\end{equation}%
\begin{equation}
\theta _{1}(z,\tau +1)=e^{\frac{\pi \sqrt{-1}}{4}}\theta _{1}(z,\tau ),\ \ \
\theta _{1}\left( z,-{1}/{\tau }\right) =\left( {\frac{\tau }{\sqrt{-1}}}%
\right) ^{1/2}e^{\pi \sqrt{-1}\tau z^{2}}\theta _{2}(\tau z,\tau )\ ;
\label{Eqn: theta_1 transformation}
\end{equation}%
\begin{equation}
\theta _{2}(z,\tau +1)=\theta _{3}(z,\tau ),\ \ \ \theta _{2}\left( z,-{1}/{%
\tau }\right) =\left( {\frac{\tau }{\sqrt{-1}}}\right) ^{1/2}e^{\pi \sqrt{-1}%
\tau z^{2}}\theta _{1}(\tau z,\tau )\ ;  \label{Eqn: theta_2 transformation}
\end{equation}%
\begin{equation}
\theta _{3}(z,\tau +1)=\theta _{2}(z,\tau ),\ \ \ \theta _{3}\left( z,-{1}/{%
\tau }\right) =\left( {\frac{\tau }{\sqrt{-1}}}\right) ^{1/2}e^{\pi \sqrt{-1}%
\tau z^{2}}\theta _{3}(\tau z,\tau )\ .  \label{Eqn: theta_3 transformation}
\end{equation}

One also has the following  formulas about how the theta functions vary
along the lattice $\Gamma =\{a+b\tau |a,b\in \mathbf{Z}\}$ (cf. \cite{C}),
\begin{equation}
\theta (z+a,\tau )=(-1)^{a}\theta (z,\tau ), \ \theta (z+b\tau ,\tau )=(-1)^{b}e^{-2\pi \sqrt{-1}bz-\pi \sqrt{-1}b^{2}\tau }\theta (z,\tau );
\end{equation}
\begin{equation}
\theta _{1}(z+a,\tau )=(-1)^{a}\theta _{1}(z,\tau ),\ \theta _{1}(z+b\tau
,\tau )=e^{-2\pi \sqrt{-1}bz-\pi \sqrt{-1}b^{2}\tau }\theta _{1}(z,\tau );
\label{Eqn: theta 1 transformation}
\end{equation}%
\begin{equation}
\theta _{2}(z+a,\tau )=\theta _{2}(z,\tau ),\ \theta _{2}(z+b\tau ,\tau
)=(-1)^{b}e^{-2\pi \sqrt{-1}bz-\pi \sqrt{-1}b^{2}\tau }\theta _{2}(z,\tau );
\label{Eqn: theta 2 transformation}
\end{equation}%
\begin{equation}
\theta _{3}(z+a,\tau )=\theta _{3}(z,\tau ),\ \theta _{3}(z+b\tau ,\tau
)=e^{-2\pi \sqrt{-1}bz-\pi \sqrt{-1}b^{2}\tau }\theta _{3}(z,\tau ).
\label{Eqn: theta 3 transformation}
\end{equation}

\subsection{The basic representation for the affine $E_8$}
In this subsection we  briefly review the basic representation for the affine $E_8$ following \cite{K1} (see also \cite{K2}).

Let $\mathfrak{g}$ be the  (complex) Lie algebra of $E_8$. Let $\langle, \rangle$ be the Killing form on $\mathfrak{g}$.
Let $\widetilde{\mathfrak{g}}$ be the affine Lie algebra corresponding to $\mathfrak{g}$ defined by
$$\widetilde{\mathfrak{g}}=\mathbf{C}[t, t^{-1}]\otimes \mathfrak{g}\oplus \CC c, $$ with bracket
$$[P(t)\otimes x+\lambda c, Q(t)\otimes y+\mu c]=P(t)Q(t)\otimes [x,y]+\langle x, y\rangle\,\mathrm{Res}_{t=0}\left(\frac{dP(t)}{dt}Q(t)\right)c.$$

Let $\widehat{\mathfrak{g}}$ be the affine Kac-Moody algebra obtained from $\widetilde{\mathfrak{g}}$ by adding a derivation $t\frac{d}{dt}$ which operates on $\mathbf{C}[t, t^{-1}]\otimes \mathfrak{g}$ in an obvious way and sends $c$ to $0$.

The basic representation $V(\Lambda_0)$ is the $\widehat{\mathfrak{g}}$-module defined by the property that there is a nonzero vector $v_0$ (highest weight vector) in $V(\Lambda_0)$ such that $cv_0=v_0, (\mathbf{C}[t]\otimes \mathfrak{g}\oplus\CC t\frac{d}{dt})v_0=0$. Setting $V_i:=\{v\in V(\Lambda_0)| t\frac{d}{dt}v=-iv\}$ gives a $\ZZ_+$-gradation by finite dimensional subspaces. Since $[\mathfrak{g},t\frac{d}{dt}]=0$, each $V_i$ is a representation of $\mathfrak{g}$. Moreover, $V_1$ is the adjoint representation of $E_8.$

Fix a basis $\{Z_i\}_{i=1}^8$ for the Cartan subalgebra. The character of the basic representation is given by
\be \mathrm{ch}(z_1, z_2,\cdots,z_8,\tau):=\sum_{i=0}^{\infty}(\mathrm{ch}V_i)(z_1, z_2,\cdots, z_8)q^i=\varphi(\tau)^{-8}\Theta_{\mathfrak{g}}(z_1, z_2,\cdots, z_8, \tau),\ee
where $\varphi(\tau)=\prod_{n=1}^\infty (1-q^n)$ so that $\eta(\tau)=q^{1/24}\varphi(\tau)$ is the Dedekind $\eta$ function; $\Theta_{\mathfrak{g}}(z_1, z_2,\cdots, z_8, \tau)$ is the theta function defined on the root lattice $Q$ by
\be \Theta_{\mathfrak{g}}(z_1, z_2,\cdots, z_8, \tau)=\sum_{\gamma\in Q}q^{|\gamma|^2/2}e^{2\pi\ii \gamma(\sum_{l=1}^8 z_l Z_l)}.\ee

It is proved in \cite{GL} (cf. \cite{Har}) that there is a basis for the $E_8$ root lattice such that
\be
\Theta_{\mathfrak{g}}(z_1, \cdots. z_8,\tau)
=\frac{1}{2}\left( \prod_{l=1}^8\theta(z_l,\tau)+\prod_{l=1}^8\theta_1(z_l,\tau)+\prod_{l=1}^8\theta_2(z_l,\tau)+\prod_{l=1}^8\theta_3(z_l,\tau)\right).
\ee

\section{$E_8$ Bundles and Rigidity}
In this section we prove two rigidity and vanishing theorems for spin$^c$ manifolds with $E_8$ principal bundles. Theorem 0.1 is deduced from the first one (Theorem 2.1). 

Let $X$ be a $2k$ dimensional closed spin$^c$ manifold, which admits a nontrivial $S^1$ action that preserves the spin$^c$ structure. Let $L$ be the complex line bundle associated with the spin$^c$ structure of $X$. It's the associated line bundle of the $U(1)$-bundle $Q/spin(2k)\rightarrow Q/spin^c(2k)\cong X$, where $Q$ is the $spin^c(2k)$ principal bundle over $X$ determined by the $spin^c$ structure. We denote the first equivariant Chern class of $L$ by $c_1(X)_{S^1}$.  Let $P$ be an $E_8$ principal bundle over $X$ such that the $S^1$ action on $X$ can be lifted to $P$ as a left action which commutes with the free action of $E_8$ on $P$. Let $W$ be the vector bundle associated to the complex adjoint representation of $E_8$ mentioned above. Then  the $S^1$ action on $P$ naturally induces an action on $W$ as described in the introduction.

Let $g^{TX}$ be a Riemannian metric on $X$. Let $\nabla^{TX}$ be the
Levi-Civita connection associated to $g^{TX}$. Denote the complexification of $TX$ by $T_\CC X$. Let $g^{T_{\mathbf{C}}X}$ and
$\nabla^{T_{\mathbf{C}}X}$ be the induced Hermitian metric and Hermitian
connection on $T_{\mathbf{C}}X$. Let $h^L$ be a Hermitian metric on $L$ and $\nabla^L$ be a Hermitian connection. Let $\overline{L}$ be the complex conjugate of $L$ with the induced Hermitian metric and connection. Assume that the $S^1$ action on $X$ preserves the metrics and connections involved. Let $S_c(TX)=S_{c,+}(TX)\oplus S_{c,-}(TX)$ denote the bundle of spinors
associated to the spin$^c$ structure, $(TX,g^{TX})$ and $(L,h^L)$. Then $%
S_c(TX)$ carries induced Hermitian metric and connection preserving the
above $\mathbf{Z}_2$-grading. Let $D_{c,\pm}:\Gamma(S_{c,\pm}(TX))%
\rightarrow \Gamma(S_{c,\mp}(TX))$ denote the induced spin$^c$ Dirac
operators (cf. \cite{LM}). If $V$ is an equivariant complex vector bundle over $X$ with equivariant Hermitian metric $h^V$ and Hermitian connection $\nabla^V$, let $D_{c,\pm}^V:\Gamma(S_{c,\pm}(TX)\otimes V)
\rightarrow \Gamma(S_{c,\mp}(TX)\otimes V)$ denote the induced twisted spin$^c$ Dirac
operators.

\begin{theorem} Assume the action only has isolated fixed points and the restriction of the equivariant characteristic class 
$$\frac{1}{30}c_2(W)_{S^1}+3c_1(X)_{S^1}^2-p_1(TX)_{S^1}$$ 
to $X^{S^1}$ is equal to $n\cdot \pi^*u^2$ for some integer $n$. \newline
(i) If $n<0$, then
$$\mathrm{Ind}(g, D_{c,+}^{(1+\overline{L})\otimes W})+\mathrm{Ind}(g, D_{c,+}^{(1+\overline{L})\otimes (T_\CC X-(L^2+\overline{L}^2)+(L+\overline{L}))})\equiv 0.$$ In particular,
 $$\mathrm{Ind}D_{c,+}^{(1+\overline{L})\otimes W}+\mathrm{Ind}D_{c,+}^{(1+\overline{L})\otimes (T_\CC X-(L^2+\overline{L}^2)+(L+\overline{L}))}=0.$$
(ii) If $n=0$, then
$$\mathrm{Ind}(g, D_{c,+}^{(1+\overline{L})\otimes W})+\mathrm{Ind}(g, D_{c,+}^{(1+\overline{L})\otimes (T_\CC X-(L^2+\overline{L}^2)+(L+\overline{L}))})$$
is independent of $g$. Moreover, when $k$ is odd, one has $$\mathrm{Ind}(g, D_{c,+}^{(1+\overline{L})\otimes W})+\mathrm{Ind}(g, D_{c,+}^{(1+\overline{L})\otimes (T_\CC X-(L^2+\overline{L}^2)+(L+\overline{L}))})\equiv 0.$$
(iii) If $n=2$ and  $k$ is odd, then $$\mathrm{Ind}(g, D_{c,+}^{(1+\overline{L})\otimes W})+\mathrm{Ind}(g, D_{c,+}^{(1+\overline{L})\otimes (T_\CC X-(L^2+\overline{L}^2)+(L+\overline{L}))})\equiv 0.$$

\end{theorem}

\begin{proof}

Let $g=e^{2\pi \sqrt{-1}t}\in S^1$ be the generator of the action group. Let $X^{S^1}=\{p\}$ be the set of fixed points. Let $TX|p=E_1\oplus \cdots\oplus E_k$ be the decomposition of the tangent bundle into the $S^1$-invariant 2-planes. Assume that $g$ acts on $E_j$ by $e^{2\pi \sqrt{-1}\alpha_j t}, \alpha_j\in \ZZ$. Assume $g$ acts on $L|_p$ by $e^{2\pi \sqrt{-1}c t}, c\in \ZZ$. Clearly, \be p_1(TM|_p)_{S^1}=(2\pi\sqrt{-1})^2\sum_{j=1}^k{\alpha_j}^2 t^2, \ c_1(L|_p)_{S^1}=2\pi\sqrt{-1} ct.\ee

Denote $L\oplus\overline{L}$ by $L_\CC$.  If $E$ is a complex vector bundle over $X$, set $\widetilde{E}=E-\mathbf{C}^{\mathrm{rk}(E)}\in
K(X)$.

Let $\Theta(X, L, \tau)$ be the virtual
complex vector bundle over $X$ defined by
\begin{equation*}
\begin{split}
\Theta(X, L, \tau):=& \left( \overset{%
\infty }{\underset{m=1}{\otimes }}S_{q^{m}}(\widetilde{T_{\mathbf{C}}X})\right) \otimes \left( \overset{\infty }{\underset{u=1}{\otimes }}\Lambda
_{q^{u}}(\widetilde{L_\CC})\right) \\
& \otimes \left( \overset{\infty }{\underset{v=1}{\otimes }}\Lambda
_{-q^{v-1/2}}(\widetilde{L_\CC})\right) \otimes
\left( \overset{\infty }{\underset{w=1}{\otimes }}\Lambda _{q^{w-1/2}}(\widetilde{L_\CC})\right) ,
\end{split}%
\end{equation*}

Let $W_i \ (i=0,1.\cdots)$ be the associated bundles $P\times_{\rho_i}V_i$, where $V_i$'s are the representations of $E_8$ as in \S 1.2. Then $W=W_1$.

Consider the twisted operator
\be D_{c,+}^{(1+\overline{L})\otimes\Theta(X, L, \tau)\otimes (\varphi^8(\tau)\sum_{i=0}^\infty W_iq^i)}.\ee Expanding $q$-series, we have
\be
\begin{split}
&\Theta(X, L, \tau)\otimes (\varphi^8(\tau)\sum_{i=0}^\infty W_iq^i)\\
=&(1+(T_\CC X-2k)q+O(q^2))\otimes(1+\widetilde{L_\CC}q+O(q^2))\\
&\otimes(1-\widetilde{L_\CC}q^{1/2}-2\widetilde{L_\CC}q+O(q^{3/2}))\otimes(1+\widetilde{L_\CC}q^{1/2}-2\widetilde{L_\CC}q+O(q^{3/2}))\\
&\otimes (1-8q+O(q^2))\otimes (1+Wq+O(q^2))\\
=&1+(W-8+T_\CC X-2k-3\widetilde{L_\CC}-\widetilde{L_\CC}\otimes\widetilde{L_\CC})q+O(q^2).
\end{split}
\ee

It's not hard to see that $\widetilde{L_\CC}\otimes\widetilde{L_\CC}=L^2+\overline{L}^2-4(L+\overline{L})+6.$
So
\be
\begin{split}
&D_{c,+}^{(1+\overline{L})\otimes\Theta(M, L, \tau)\otimes (\varphi^8(\tau)\sum_{i=0}^\infty W_iq^i)}\\
=& D_{c,+}^{(1+\overline{L})}+D_{c,+}^{(1+\overline{L})\otimes(W+T_\CC X-(L^2+\overline{L}^2)+(L+\overline{L})-8-2k)}q+O(q^2).
\end{split}
\ee

By the Atiyah-Bott-Segal-Singer Letschetz fixed point formula, for the twisted operator $ D_{c,+}^{(1+\overline{L})\otimes\Theta(X, L, \tau)\otimes (\varphi^8(\tau)\sum_{i=0}^\infty W_iq^i)}$, the equivariant index
\be
\begin{split}
I(t, \tau)=&2\sum_p \left\{\frac{1}{(2\pi \sqrt{-1})^k}\prod_{j=1}^k\frac{\theta'(0, \tau)}{\theta(\alpha_j t,\tau)}
\frac{\theta_1(ct,\tau)}{\theta_1(0,\tau)}\frac{\theta_2(ct,\tau)}{\theta_2(0,\tau)}\frac{\theta_3(ct,\tau)}{\theta_3(0,\tau)}\right.\\
 &\ \ \ \ \ \ \ \ \ \ \left.\cdot \varphi^8(\tau)\cdot \left(\sum_{i=0}^\infty \mathrm{ch}(W_i|_p)_{S^1}q^i\right)\right\}.
\end{split}
\ee

On the fixed point $p$, fixing an element $s\in P|p$, one can define a map $f_{s}: S^1\to E_8$ by $g\cdot s=s\cdot f_{s}(g)$. It's not hard to check that $f_{s}$ is a group homomorphism. Moreover, for $h\in E_8$, we have
$$g\cdot (s\cdot h)=(g\cdot s)\cdot h=s\cdot f_{s}(g)\cdot h=(s\cdot h)\cdot (h^{-1}f_s(g)h).$$ As all the maximal tori in $E_8$ are conjugate, then one may choose $s\in P|p$ such that $f_s: S^1\to E_8$ maps $S^1$ into the maximal torus $\mathfrak{t}$ that corresponds to the Cartan subalgebra such that the theta function $\Theta_\mathfrak{g}(z_1, \cdots, z_8, \tau)$ appears as in (1.16).
For any unitary representation $\rho: E_8\to U(N)$, let $\mathfrak{T}$ be a maximal torus of $U(N)$ that contains $\rho(\mathfrak{t})$. Let
$$\xymatrix@C=0.5cm{\widehat{\mathfrak{T}}\ar[r]^{\widehat{\rho}}& \widehat{\mathfrak{t}}\ar[r]^{\widehat{f_s}} &\widehat{S^1}}$$ be the induced maps on the character groups. Assume $\widehat{f_s}(z_i)=\beta_i t$. Let $\{x_i\}$ are basis for $\widehat{\mathfrak{T}}$. By definition,
$$(\mathrm{ch}\rho)(z_1, z_2,\cdots, z_8)=\sum_{i=1}^N e^{\widehat{\rho}(x_i)},$$ and therefore
\h
\begin{split}
&(\mathrm{ch}\rho)(\beta_1t, \beta_2 t,\cdots, \beta_8 t)\\
=&\widehat{f_s}((\mathrm{ch}\rho)(z_1, z_2,\cdots, z_8))\\
=&\sum_{i=1}^N e^{(\widehat{f_s}\circ\widehat{\rho})(x_i)}\\
=&\mathrm{ch}((P\times_\rho \CC^N)|_p)_{S^1}.
\end{split}
\e
So for each $i$, we have $ \mathrm{ch}(W_i|p)_{S^1}=(\mathrm{ch}V_i)(\beta_1t, \beta_2t,\cdots, \beta_8t)$. Then by (1.14) and (1.16), we have
\be
\begin{split}
&\varphi^8(\tau)\cdot \left(\sum_{i=0}^\infty \mathrm{ch}(W_i|_p)_{S^1}q^i\right)\\
=&\frac{1}{2}\left( \prod_{l=1}^8\theta(\beta_l t,\tau)+\prod_{l=1}^8\theta_1(\beta_l t,\tau)+\prod_{l=1}^8\theta_2(\beta_l t,\tau)+\prod_{l=1}^8\theta_3(\beta_l t,\tau)\right).
\end{split}
\ee
Comparing both sides of (2.6), we can see by direct computation that
\be 30\cdot (2\pi\sqrt{-1})^2\sum_{l=1}^8\beta_l^2t^2=c_2(W|_p)_{S^1}.\ee

By (2.5) and (2.6), we have
\be
\begin{split}
I(t, \tau)=&\sum_p \left\{\frac{1}{(2\pi \sqrt{-1})^k}\prod_{j=1}^k\frac{\theta'(0, \tau)}{\theta(\alpha_j t,\tau)}
\frac{\theta_1(ct,\tau)}{\theta_1(0,\tau)}\frac{\theta_2(ct,\tau)}{\theta_2(0,\tau)}\frac{\theta_3(ct,\tau)}{\theta_3(0,\tau)}\right.\\
 &\ \ \ \ \ \ \ \ \ \left.\cdot \left( \prod_{l=1}^8\theta(\beta_l t,\tau)+\prod_{l=1}^8\theta_1(\beta_l t,\tau)+\prod_{l=1}^8\theta_2(\beta_l t,\tau)+\prod_{l=1}^8\theta_3(\beta_l t,\tau)\right)\right\}.
\end{split}
\ee

From the transformation laws of theta functions (1.10)-(1.13), for $a, b\in 2\ZZ$, it's not hard to see that
$$ I(t+a\tau+b, \tau)=e^{-\pi\sqrt{-1}(\sum_{l=1}^8\beta_l^2+3c^2-\sum_{j=1}^km_j^2)(b^2\tau+2b\tau)}I(t, \tau).
$$
Since when restricted to fixed points, $\frac{1}{30}c_2(W)_{S^1}+3c_1(L)_{S^1}^2-p_1(TX)_{S^1}$ is equal to $n\cdot \pi^*u^2$, then for each fixed point, from (2.1) and (2.7) we have
$$\sum_{l=1}^8\beta_l^2+3c^2-\sum_{j=1}^k\alpha_j^2=n$$ and therefore
\be I(t+a\tau+b, \tau)=e^{-\pi\sqrt{-1}n(b^2\tau+2b\tau)}I(t, \tau). \ee

It's easy to deduce from (1.6) that
$$
\theta'(0,\tau +1)=e^{\frac{\pi \sqrt{-1}}{4}}\theta'(0,\tau ),\ \ \ \theta'
\left( 0,-{1}/{\tau }\right) ={\frac{1}{\sqrt{-1}}}\left( {\frac{\tau }{%
\sqrt{-1}}}\right) ^{1/2}\tau\theta' \left(0,\tau \right).
$$
Using the above two formulas and the transformation laws of theta functions (1.6)-(1.9), we have

\be I(t, \tau+1)=I(t, \tau)
\ee
and 
\be I\left(\frac{t}{\tau}, -\frac{1}{\tau}\right)=\tau^{k+4}e^{\frac{\pi\sqrt{-1}\left(\sum_{l=1}^8\beta_l^2+3c^2-\sum_{j=1}^k\alpha_j^2\right)t^2}{\tau}}I(t, \tau)=\tau^{k+4}e^{\frac{\pi\sqrt{-1}nt^2}{\tau}}I(t, \tau).
\ee
(2.9)-(2.11) tell us that $I(t, \tau)$ obeys the transformation laws that a Jacobi form (see \cite{EZ}) should satisfy.

Next we shall prove that $I(t, \tau)$ is holomorphic for $(t,\tau)\in \CC\times \HH$. First, we have the following lemma:
\begin{lemma} $I(t, \tau)$ is holomorphic for $(t,\tau)\in \RR\times \HH$.

\end{lemma}
\noindent The proof of this lemma is almost verbatimly same as the proof of Lemma 1.3 in \cite{Liu1}. We shall prove that $I(t, \tau)$ is actually holomorphic on $\CC\times \HH$. The possible polar divisor of $I(t,\tau)$ can be written in the form $t=\frac{m(c\tau+d)}{l}$ for integers $m, l, c, d$ with $(c,d)=1$. Assume $\frac{m(c\tau+d)}{l}$ is a pole for $I(t, \tau)$. Find integers $a,b$ such that $ad-bc=1$. Consider the function $I\left(\frac{t}{-c\tau+a}, \frac{d\tau-b}{-c\tau+a}\right)$. By (2.10) and (2.11), it's easy to see that
\be
I\left(\frac{t}{-c\tau+a}, \frac{d\tau-b}{-c\tau+a}\right)=f(t,\tau)\cdot I(t, \tau),
\ee
where $f(t,\tau)$ is an entire function of $t$ for every $\tau\in \HH$. If $\tau'=\frac{a\tau+b}{c\tau+d}$, then $\tau=\frac{d\tau'-b}{-c\tau'+a}$ and $\frac{m\left( c\frac{d\tau'-b}{-c\tau'+a}+d\right)}{l}$ is a pole for the function $I\left(t, \frac{d\tau'-b}{-c\tau'+a}\right)$.
However by (2.12), we have
\h
\begin{split}
&I\left(\frac{m\left( c\frac{d\tau'-b}{-c\tau'+a}+d\right)}{l}, \frac{d\tau'-b}{-c\tau'+a}\right)\\
=&I\left(\frac{\frac{m}{l}}{-c\tau'+a},  \frac{d\tau'-b}{-c\tau'+a}\right)\\
=&f\left(\frac{m}{l},\tau'\right)\cdot I\left(\frac{m}{l},\tau'\right).
\end{split}
  \e
As $\frac{m}{l}$ is real, by Lemma 2.1, we get a contradiction. Therefore $I(t, \tau)$ is holomorphic for $(t,\tau)\in \CC\times \HH.$

Combining the transformation formulas (2.9)-(2.11) and the holomorphicity of $I(t,\tau)$ on $\CC\times \HH$,  we see that $I(t, \tau)$ is a weak Jacobi form of index $\frac{n}{2}$ and weight $k+4$ over $(2\ZZ)^2\rtimes SL(2, \ZZ)$. Here by weak Jacobi form, we don't require the regularity condition at the cusp but only require that at the cusp $q$ appears with nonnegative powers only. We refer to \cite{EZ} for the precise definition of the Jacobi forms.

If $n=0$, by (2.9), we see that $I(t, \tau)$ is holomorphic on the torus $$\CC/2\ZZ+2\ZZ\tau$$ and therefore must be independent of $t$.  So, by (2.4), we see that
$$\mathrm{Ind}(g, D_{c,+}^{(1+\overline{L})}),$$
$$\mathrm{Ind}(g, D_{c,+}^{(1+\overline{L})\otimes(W+T_\CC X-(L^2+\overline{L}^2)+(L+\overline{L})-8-2k)})$$
are both independent of $g$. So
$$\mathrm{Ind}(g, D_{c,+}^{(1+\overline{L})\otimes W})+\mathrm{Ind}(g, D_{c,+}^{(1+\overline{L})\otimes (T_\CC X-(L^2+\overline{L}^2)+(L+\overline{L}))})$$
must be independent of $g$. The index density of the operator $$D_{c,+}^{(1+\overline{L})\otimes W}+D_{c,+}^{(1+\overline{L})\otimes (T_\CC X-(L^2+\overline{L}^2)+(L+\overline{L}))}$$ involves the characteristic forms 
$$\widehat{A}(TM), e^{c_1(L)/2}(1+e^{-c_1(L)}), \mathrm{ch}(W), \mathrm{ch}(T_\CC M), \mathrm{ch}(L+\overline{L}), \mathrm{ch}(L^2+\overline{L}^2), $$
which are all of degree $4l$ (noting that $W$ is the complexification of the real adjoint representation of compact $E_8$). Therefore by the Atiyah-Singer index theorem,
$\mathrm{Ind}D_{c,+}^{(1+\overline{L})\otimes W}+\mathrm{Ind}D_{c,+}^{(1+\overline{L})\otimes (T_\CC X-(L^2+\overline{L}^2)+(L+\overline{L}))}$ (i.e. when $g=id$) must be 0 when the dimension of the manifold is not divisible by 4.  So when $k$ is odd,
$$\mathrm{Ind}(g, D_{c,+}^{(1+\overline{L})\otimes W})+\mathrm{Ind}(g, D_{c,+}^{(1+\overline{L})\otimes (T_\CC X-(L^2+\overline{L}^2)+(L+\overline{L}))})\equiv 0.$$ This finishes the proof of part (ii).

If $n\neq 0$, i.e in the case of nonzero anomaly, we need the following two lemmas.
\begin{lemma}[\protect Theorem 1.2 in \cite{EZ}] Let $I$ be a weak Jacobi form of index $m$ and weight $h$. Then for fixed $\tau$, if not identically 0, $I$ has exactly 2m zeros in any fundamental domain for the action of the lattice on $\CC$.

\end{lemma}

\begin{lemma}[\protect Theorem 2.2 in \cite{EZ}] Let $I$ be a weak Jacobi form of index $m$ and weight $h$. If $m=1$ and $h$ is odd, then $I$ is identically 0.

\end{lemma}

We would like to point that Lemma 2.2 and Lemma 2.3 are stated in \cite{EZ} for Jacobi forms. However, as in the proofs of them no regularity condition at the cusp are used, we state them here for weak Jacobi forms. See \cite{EZ} for details. 

If $n<0$, then by Lemma 2.2, $I(t, \tau)\equiv 0$, therefore
$$\mathrm{Ind}(g, D_{c,+}^{(1+\overline{L})})\equiv 0,$$
$$\mathrm{Ind}(g, D_{c,+}^{(1+\overline{L})\otimes(W+T_\CC X-(L^2+\overline{L}^2)+(L+\overline{L})-8-2k)})\equiv 0.$$
So part (i) follows.

If $n=2$, as the the the weight of $I(t,\tau)$ is $k+4$, so part (iii) similarly follows clearly from Lemma 2.3.

\end{proof}

Theorem 0.1 can be easily deduced from Theorem 2.1 as follows. 

$ \, $

\noindent {\it Proof of Theorem 0.1:} When $X$ is a spin manifold, $L$ is trivial and $D_{c,+}=D$.  By the Atiyah-Hirzebruch vanishing theorem (\cite{AH1}), we have $\mathrm{Ind}(g, D)\equiv 0$. Moreover by the Witten rigidity theorem (\cite{T, BT, Liu1}, the operator $D^{T_\CC X}$ is rigid. i.e. $\mathrm{Ind}(g, D^{T_\CC X})\equiv \mathrm{Ind}D^{T_\CC X}$. Also note that $\mathrm{Ind}D^{T_\CC X}$ equals to 0 when $k$ is odd. Then the three parts in Theorem 0.1 easily follow from the corresponding three parts in Theorem 2.1. $\square$

$\, $

For Spin$^c$ manifolds, we have rigidity and vanishing theorem for another type of twisted operators.
\begin{theorem}Assume the action only has isolated fixed points and the restriction of the equivariant characteristic class 
$$\frac{1}{30}c_2(W)_{S^1}+c_1(X)_{S^1}^2-p_1(TX)_{S^1}$$ 
to $X^{S^1}$ is equal to $n\cdot \pi^*u^2$ for some integer $n$. \newline
(i) If $n<0$, then
$$\mathrm{Ind}(g, D_{c,+}^{(1-\overline{L})\otimes W})+\mathrm{Ind}(g, D_{c,+}^{(1-\overline{L})\otimes (T_\CC X-(L+\overline{L}))})\equiv 0.$$
In particular,
 $$\mathrm{Ind}D_{c,+}^{(1-\overline{L})\otimes W}+\mathrm{Ind} D_{c,+}^{(1-\overline{L})\otimes (T_\CC X-(L+\overline{L}))}=0.$$
(ii) If $n=0$, then
$$\mathrm{Ind}(g, D_{c,+}^{(1-\overline{L})\otimes W})+\mathrm{Ind}(g, D_{c,+}^{(1-\overline{L})\otimes (T_\CC X-(L+\overline{L}))}) $$
is independent of $g$. Moreover, when $k$ is even, one has $$\mathrm{Ind}(g, D_{c,+}^{(1-\overline{L})\otimes W})+\mathrm{Ind}(g, D_{c,+}^{(1-\overline{L})\otimes (T_\CC X-(L+\overline{L}))}) \equiv 0.$$
(iii) If $n=2$ and  $k$ is even, then $$\mathrm{Ind}(g, D_{c,+}^{(1-\overline{L})\otimes W})+\mathrm{Ind}(g, D_{c,+}^{(1-\overline{L})\otimes (T_\CC X-(L+\overline{L}))})\equiv 0. $$

\end{theorem}

\begin{proof}
We will use same notations as in the proof of Theorem 2.1.

Let $\Theta^{\ast }(X, L, \tau)$ be the virtual
complex vector bundles over $X$ defined by
\begin{equation*}
\Theta ^{\ast }(X, L, \tau):=\left(
\overset{\infty }{\underset{m=1}{\otimes }}S_{q^{m}}(\widetilde{T_{\mathbf{C}}X})\right) \otimes \left( \overset{\infty }{\underset{u=1}{\otimes }}%
\Lambda _{-q^{u}}(\widetilde{L_\CC})\right) .
\end{equation*}

Consider the twisted operator
\be D_{c,+}^{(1-\overline{L})\otimes\Theta^*(X, L, \tau)\otimes (\varphi^8(\tau)\sum_{i=0}^\infty W_iq^i)}.\ee Expanding $q$-series, we have
\be
\begin{split}
&\Theta^*(X, L, \tau)\otimes (\varphi^8(\tau)\sum_{i=0}^\infty W_iq^i)\\
=&(1+(T_\CC X-2k)q+O(q^2))\otimes(1-\widetilde{L_\CC}q+O(q^2))\\
&\otimes (1-8q+O(q^2))\otimes (1+Wq+O(q^2))\\
=&1+(W+T_\CC X-(L+\overline{L})-2k-6)q+O(q^2).
\end{split}
\ee
So
\be
\begin{split}
&D_{c,+}^{(1-\overline{L})\otimes\Theta^*(X, L, \tau)\otimes (\varphi^8(\tau)\sum_{i=0}^\infty W_iq^i)}\\
=& D_{c,+}^{(1-\overline{L})}+D_{c,+}^{(1-\overline{L})\otimes(W+T_\CC X-(L+\overline{L})-2k-6) }q+O(q^2).
\end{split}
\ee

By the Atiyah-Bott-Segal-Singer Letschetz fixed point formula, for this twisted operator $ D_{c,+}^{(1-\overline{L})\otimes\Theta^*(X, L, \tau)\otimes (\varphi^8(\tau)\sum_{i=0}^\infty W_iq^i)}$, the equivariant index
\be
\begin{split}
J(t, \tau)=&2\sum_p \left\{\frac{1}{(2\pi \sqrt{-1})^k}\prod_{j=1}^k\frac{\theta'(0, \tau)}{\theta(\alpha_j t,\tau)}
\frac{\theta(ct, \tau)}{\theta_1(0,\tau)\theta_2(0,\tau)\theta_3(0,\tau)}\right.\\
 &\ \ \ \ \ \ \ \ \ \ \ \ \left.\cdot \varphi^8(\tau)\cdot \left(\sum_{i=0}^\infty \mathrm{ch}(W_i|_p)_{S^1}q^i\right)\right\}\\
 =&\sum_p \left\{\frac{1}{(2\pi \sqrt{-1})^k}\prod_{j=1}^k\frac{\theta'(0, \tau)}{\theta(\alpha_j t,\tau)}
\frac{\theta(ct, \tau)}{\theta_1(0,\tau)\theta_2(0,\tau)\theta_3(0,\tau)}\right.\\
&\ \ \ \ \ \ \ \ \ \left.\cdot \left( \prod_{l=1}^8\theta(\beta_l t,\tau)+\prod_{l=1}^8\theta_1(\beta_l t,\tau)+\prod_{l=1}^8\theta_2(\beta_l t,\tau)+\prod_{l=1}^8\theta_3(\beta_l t,\tau)\right)\right\}.
\end{split}
\ee

 As when restricted to fixed points, $\frac{1}{30}c_2(W)_{S^1}+c_1(L)_{S^1}^2-p_1(TX)_{S^1}$ is equal to $n\cdot \pi^*u^2$, then for each fixed point, we have
$$\sum_{l=1}^8\beta_l^2+c^2-\sum_{j=1}^k\alpha_j^2=n.$$ Therefore, similar to (2.9), one can show that for $a, b\in 2\ZZ$
\be J(t+a\tau+b, \tau)=e^{-\pi\sqrt{-1}n(b^2\tau+2b\tau)}J(t, \tau). \ee
One can also show that
\be J(t, \tau+1)=J(t, \tau)
\ee
and
\be J\left(\frac{t}{\tau}, -\frac{1}{\tau}\right)=\tau^{k+3}e^{\frac{\pi\sqrt{-1}nt^2}{\tau}}J(t, \tau).
\ee

So similar to $I(t, \tau)$ in the proof of Theorem 2.1, combing Lemma 2.1 and the above transformation laws, we can prove that $J(t, \tau)$ is a weak Jacobi form of index $\frac{n}{2}$ and weight $k+3$ over $(2\ZZ)^2\rtimes SL(2, \ZZ)$.

Then one can prove the three parts of Theorem 2.2 almost the same as those in Theorem 2.1. The only difference one needs to notice is that by the Atiyah-Singer index theorem,
$\mathrm{Ind}D_{c,+}^{(1-\overline{L})\otimes W}+\mathrm{Ind} D_{c,+}^{(1-\overline{L})\otimes (T_\CC X-(L+\overline{L}))}$ must be 0 when the dimension of the manifold is divisible by 4 as the index density of the operator
$$D_{c,+}^{(1-\overline{L})\otimes W}+D_{c,+}^{(1-\overline{L})\otimes (T_\CC X-(L+\overline{L}))} $$
is a differential form of degree $4l+2$.
\end{proof}
$$ $$

\noindent {\bf Acknowledgements.} The first author is partially supported by the Academic Research Fund R-146-000-163-112  from National University of
Singapore. The second author is partially supported by NSF. The
third author is partially supported by  MOEC and NNSFC.

\end{document}